\documentclass[11pt,a4paper]{article}
\usepackage{epsf,epsfig,amsfonts,amsgen,amsmath,amstext,amsbsy,amsopn,amsthm
}
\usepackage{ebezier,eepic}
\usepackage{color}
\newtheorem{theorem}{Theorem}

\newtheorem{lemma}{Lemma}
\newtheorem{false statement}{False statement}

\theoremstyle{definition}

\newtheorem{claim}{Claim}
\newtheorem{subclaim}{Claim}[claim]

\newtheorem{problem}{Problem}
\newtheorem{case}{Case}
\newtheorem{subcase}{Case}[case]
\newtheorem{subsubcase}{Case}[subcase]

\newcounter{mathitem}
\newenvironment{mathitem}
  {\begin{list}{{$(\roman{mathitem})$}}{
   \setcounter{mathitem}{0}
   \usecounter{mathitem}
   \setlength{\topsep}{0pt plus 2pt minus 0pt}
   \setlength{\parskip}{0pt plus 2pt minus 0pt}
   \setlength{\partopsep}{0pt plus 2pt minus 0pt}
   \setlength{\parsep}{0pt plus 2pt minus 0pt}
   \setlength{\leftmargin}{35pt}
   \setlength{\itemsep}{0pt plus 2pt minus 0pt}}}
  {\end{list}}

\newcommand{\de}{{\rm def}}

\begin{document}

\title{\bf\Large Pairs of heavy subgraphs for Hamiltonicity of 2-connected graphs\thanks{This paper was supported  by NSFC
(No.~10871158).}}

\date{}

\author{Binlong Li$^a$,
Zden\v{e}k Ryj\'a\v{c}ek$^{b,}$\thanks{Research partially supported
by grant 1M0545 of the Czech Ministry of Education},
Ying Wang$^a$ and Shenggui Zhang$^{a,}$\thanks{Corresponding author.
E-mail address: sgzhang@nwpu.edu.cn (S. Zhang).}\\[2mm]
\small $^a$ Department of Applied Mathematics,
\small Northwestern Polytechnical University,\\
\small Xi'an, Shaanxi 710072, P.R.~China\\
\small $^b$ Department of Mathematics, University of West Bohemia and \\
\small Institute for Theoretical Computer Science, Charles University,\\
\small 30614 Pilsen, Czech Republic}


\maketitle

\begin{abstract}
Let $G$ be a graph on $n$ vertices. An induced subgraph $H$ of $G$
is called heavy if there exist two nonadjacent vertices in $H$ with
degree sum at least $n$ in $G$. We say that $G$ is $H$-heavy if every
induced subgraph of $G$ isomorphic to $H$ is heavy.  For a family
$\mathcal{H}$ of graphs, $G$ is called $\mathcal{H}$-heavy if $G$ is
$H$-heavy for every $H\in\mathcal{H}$. In this paper we characterize
all connected graphs $R$ and $S$ other than $P_3$ (the path on three
vertices) such that every 2-connected $\{R,S\}$-heavy graph is
Hamiltonian. This extends several previous results on
forbidden subgraph conditions for Hamiltonian graphs.

\smallskip
\end{abstract}

\section{Introduction}

We use Bondy and Murty \cite{Bondy_Murty} for terminology and
notation not defined here and consider finite simple graphs only.

Let $G$ be a graph. For a vertex $v\in V(G)$ and a subgraph $H$ of
$G$, we use $N_H(v)$ to denote the set, and $d_H(v)$ the number, of
neighbors of $v$ in $H$, respectively. We call $d_H(v)$ the
\emph{degree} of $v$ in $H$. For $x,y\in V(G)$, an
$(x,y)$-\emph{path} is a path $P$ connecting $x$ and $y$; the vertex
$x$ will be called the {\em origin} and $y$ the {\em terminus} of
$P$. For $X,Y\subset V(G)$, an $(X,Y)$-\emph{path} is a path having
its origin in $X$ and terminus in $Y$. If $x,y\in V(H)$, the
\emph{distance} between $x$ and $y$ in $H$, denoted $d_H(x,y)$, is
the length of a shortest $(x,y)$-path in $H$. When no confusion
occurs, we will denote $N_G(v)$, $d_G(v)$ and $d_G(x,y)$ by $N(v)$,
$d(v)$ and $d(x,y)$, respectively.

Let $G$ be a graph on $n$ vertices. If a subgraph $G'$ of $G$
contains all edges $xy\in E(G)$ with $x,y\in V(G')$, then $G'$ is
called an \emph{induced subgraph} of $G$. For a given graph $H$, we
say that $G$ is $H$-\emph{free} if $G$ does not contain an induced
subgraph isomorphic to $H$. For a family $\mathcal{H}$ of graphs,
$G$ is called $\mathcal{H}$-free if $G$ is $H$-free for every
$H\in\mathcal{H}$. If $H$ is an induced subgraph of $G$, we say that
$H$ is \emph{heavy} if there are two nonadjacent vertices in $V(H)$
with degree sum at least $n$ in $G$. The graph $G$ is called
$H$-\emph{heavy} if every induced subgraph of $G$ isomorphic to $H$
is heavy. For a family $\mathcal{H}$ of graphs, $G$ is called
$\mathcal{H}$-\emph{heavy} if $G$ is $H$-heavy for every
$H\in\mathcal{H}$. Note that an $H$-free graph is also $H$-heavy.

The graph $K_{1,3}$ is called the \emph{claw}, its (only) vertex of
degree 3 is called its \emph{center} and the other vertices are the
\emph{end vertices}. In this paper, instead of $K_{1,3}$-free
($K_{1,3}$-heavy), we use the terminology claw-free (claw-heavy).

The following
characterization of pairs of forbidden subgraphs for the existence
of Hamilton cycles in graphs is well known.

\begin{theorem}[Bedrossian \cite{Bedrossian}]
Let $R$ and $S$ be connected graphs with $R,S\neq P_{3}$ and let $G$
be a 2-connected graph. Then $G$ being $\{R,S\}$-free implies $G$ is
Hamiltonian if and only if (up to symmetry) $R=K_{1,3}$ and
$S=P_4,P_5,P_6,C_3,Z_1,Z_2,B,N$ or $W$ (see Fig.~1).
\end{theorem}

\begin{center}
\begin{picture}(360,200)

\thicklines

\put(5,140){\multiput(20,30)(50,0){5}{\put(0,0){\circle*{6}}}
\put(20,30){\line(1,0){100}} \put(170,30){\line(1,0){50}}
\qbezier[4](120,30)(145,30)(170,30) \put(20,35){$v_1$}
\put(70,35){$v_2$} \put(120,35){$v_3$} \put(170,35){$v_{i-1}$}
\put(220,35){$v_i$} \put(115,10){$P_i$}}

\put(265,130){\put(20,30){\circle*{6}} \put(70,30){\circle*{6}}
\put(45,55){\circle*{6}} \put(20,30){\line(1,0){50}}
\put(20,30){\line(1,1){25}} \put(70,30){\line(-1,1){25}}
\put(40,10){$C_3$}}

\put(0,0){\put(20,30){\circle*{6}} \put(70,30){\circle*{6}}
\multiput(45,55)(0,25){4}{\put(0,0){\circle*{6}}}
\put(20,30){\line(1,0){50}} \put(20,30){\line(1,1){25}}
\put(70,30){\line(-1,1){25}} \put(45,55){\line(0,1){25}}
\put(45,105){\line(0,1){25}} \qbezier[4](45,80)(45,92.5)(45,105)
\put(50,80){$v_1$} \put(50,105){$v_{i-1}$} \put(50,130){$v_i$}
\put(40,10){$Z_i$}}

\put(90,0){\put(45,40){\circle*{6}} \put(45,40){\line(-1,1){25}}
\put(45,40){\line(1,1){25}} \put(20,65){\line(1,0){50}}
\multiput(20,65)(50,0){2}{\multiput(0,0)(0,30){2}{\put(0,0){\circle*{6}}}
\put(0,0){\line(0,1){30}}} \put(25,10){$B$ (Bull)}}

\put(180,0){\multiput(20,30)(50,0){2}{\multiput(0,0)(0,30){2}{\put(0,0){\circle*{6}}}
\put(0,0){\line(0,1){30}}}
\multiput(45,85)(0,30){2}{\put(0,0){\circle*{6}}}
\put(45,85){\line(0,1){30}} \put(20,60){\line(1,0){50}}
\put(20,60){\line(1,1){25}} \put(70,60){\line(-1,1){25}}
\put(25,10){$N$ (Net)}}

\put(270,0){\put(45,30){\circle*{6}} \put(20,55){\line(1,0){50}}
\put(45,30){\line(1,1){25}} \put(45,30){\line(-1,1){25}}
\multiput(20,55)(0,30){2}{\put(0,0){\circle*{6}}}
\multiput(70,55)(0,30){3}{\put(0,0){\circle*{6}}}
\put(20,55){\line(0,1){30}} \put(70,55){\line(0,1){60}}
\put(10,10){$W$ (Wounded)}}

\end{picture}

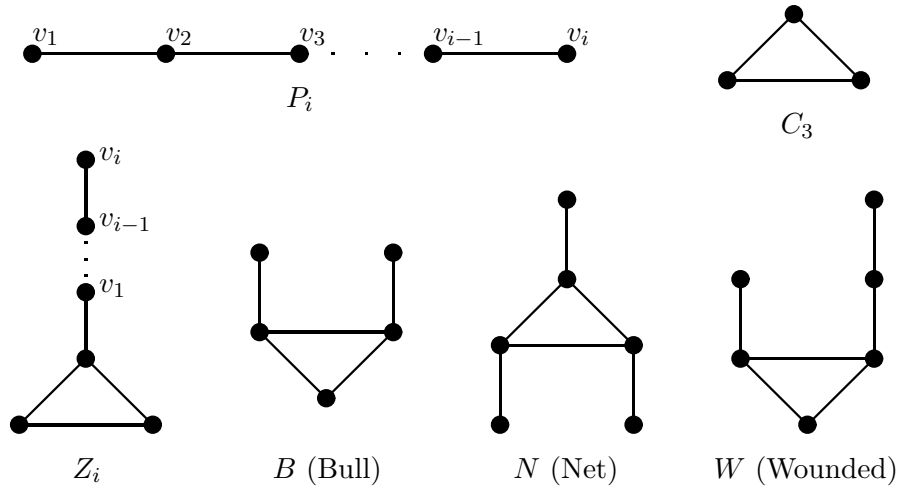
\small Fig. 1. Graphs $P_i,C_3,Z_i,B,N$ and $W$.
\end{center}

Our aim in this paper is to consider the corresponding heavy subgraph condition for
a graph to be Hamiltonian. First, we notice that every 2-connected
$P_3$-heavy graph contains a Hamilton cycle. This can be easily
deduced from the following result.

\begin{theorem}[Fan \cite{Fan}]
Let $G$ be a 2-connected graph. If $\max\{d(u),d(v)\}\geq n/2$ for
every pair of vertices with distance 2 in $G$, then $G$ is
Hamiltonian.
\end{theorem}

It is not difficult to see that $P_3$ is the only connected graph
$S$ such that every 2-connected $S$-heavy graph is Hamiltonian. So
we have the following problem.

\begin{problem}
\label{Probl1}
Which two connected graphs $R$ and $S$ other than $P_3$ imply that
every 2-connected $\{R,S\}$-heavy graph is Hamiltonian?
\end{problem}

By Theorem 1, we get that (up to symmetry) $R=K_{1,3}$ and $S$
must be some of the graphs $P_4,P_5,P_6,$ $C_3,Z_1,Z_2,B,N$ or $W$.

In this paper we prove the following results.

\begin{theorem}
If $G$ is a 2-connected $\{K_{1,3},W\}$-heavy graph, then $G$ is
Hamiltonian.
\end{theorem}

\begin{theorem}
If $G$ is a 2-connected $\{K_{1,3},N\}$-heavy graph, then $G$ is
Hamiltonian.
\end{theorem}

At the same time, we find a 2-connected $\{K_{1,3},P_6\}$-heavy
graph which is not Hamiltonian (see Fig. 2).

\begin{center}
\begin{picture}(300,160)

\put(80,70){\circle{120}} \put(80,10){\line(1,0){120}}
\put(80,130){\line(1,0){120}} \qbezier(200,10)(140,130)(80,130)
\qbezier(200,70)(140,10)(80,10) \qbezier(200,70)(140,130)(80,130)
\qbezier(200,130)(140,10)(80,10)

\thicklines

\multiput(200,10)(0,60){3}{\put(0,0){\circle*{6}}
\put(60,0){\circle*{6}} \put(30,20){\circle*{6}}
\put(0,0){\line(1,0){60}} \put(0,0){\line(3,2){30}}
\put(60,0){\line(-3,2){30}}} \put(260,10){\line(0,1){120}}
\qbezier(260,10)(300,70)(260,130)

\put(75,65){$K_r$} \put(205,120){$x_1$} \put(205,60){$x_2$}
\put(205,0){$x_3$} \put(225,140){$y_1$} \put(225,80){$y_2$}
\put(225,20){$y_3$} \put(245,120){$z_1$} \put(245,60){$z_2$}
\put(245,0){$z_3$}

\end{picture}

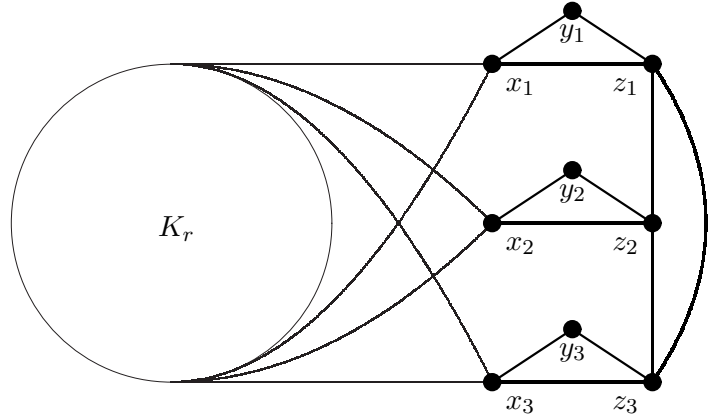
\small Fig. 2. A 2-connected $\{K_{1,3},P_6\}$-heavy non-Hamiltonian
graph ($r\geq 5$).
\end{center}

Besides, we can also construct a 2-connected claw-free and
$P_6$-heavy graph which is not Hamiltonian. This can be shown as
follows: Let $G$ be the graph in Fig. 2, where $r\geq 15$ is an
integer divisible by 3. Let $V_1,V_2,V_3$ be a balanced partition of
$K_r$ and $G'$ be the graph obtained from $G$ by deleting all the
edges in $\bigcup_{i=1}^3\{x_iv:v\in V_i\}$. Then $G'$ is a 2-connected
claw-free and $P_6$-heavy graph which is not Hamiltonian.

Note that $W$ contains induced $P_4,P_5,C_3,Z_1,Z_2$ and $B$. So we
have

\begin{theorem}
\label{Thm5}
Let $R$ and $S$ be connected graphs with $R,S\neq P_3$ and let $G$
be a 2-connected graph. Then $G$ being $\{R,S\}$-heavy implies $G$
is Hamiltonian if and only if (up to symmetry) $R=K_{1,3}$ and
$S=P_4,P_5,C_3,Z_1,Z_2,B,N$ or $W$.
\end{theorem}

Thus, Theorem~\ref{Thm5} gives a complete answer to Problem~\ref{Probl1}.

{For claw-heavy graphs, Chen \emph{et al.} get the following
result.}

\begin{theorem}[Chen, Zhang and Qiao \cite{Chen_Zhang_Qiao}]
Let $G$ be a 2-connected graph. If $G$ is claw-heavy and moreover,
$\{P_{7},D\}$-free or $\{P_{7},H\}$-free, then $G$ is Hamiltonian
(see Fig. 3).
\end{theorem}

\begin{center}
\begin{picture}(180,130)

\thicklines

\put(0,0){\put(45,30){\circle*{6}} \put(45,30){\line(-1,1){25}}
\put(45,30){\line(1,1){25}} \put(20,55){\line(1,0){50}}
\multiput(20,55)(50,0){2}{\multiput(0,0)(0,30){3}{\put(0,0){\circle*{6}}}
\put(0,0){\line(0,1){60}}} \put(25,10){$D$ (Deer)}}

\put(90,0){\multiput(20,35)(0,75){2}{\multiput(0,0)(50,0){2}{\put(0,0){\circle*{6}}}
\put(0,0){\line(1,0){50}}} \put(45,72.5){\circle*{6}}
\put(20,35){\line(2,3){50}} \put(70,35){\line(-2,3){50}}
\put(10,10){$H$ (Hourglass)}}

\end{picture}

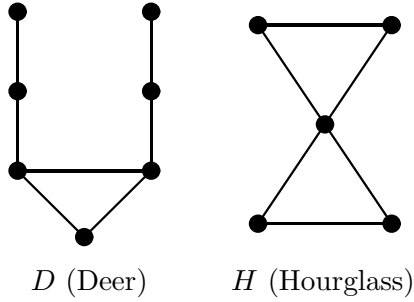
\small Fig. 3. Graphs $D$ and $H$.
\end{center}

{It is clear that every $P_6$-free graph is also
$\{P_7,D\}$-free. Thus we have that every 2-connected claw-heavy and
$P_6$-free graph is Hamiltonian.} Together with Theorems 3 and 4, we
have the following characterization:

\begin{theorem}
Let $S$ be a connected graph with $S\neq P_{3}$ and let $G$ be a
2-connected claw-heavy graph. Then $G$ being $S$-free implies $G$ is
Hamiltonian if and only if
$S=P_{4},P_{5},P_{6},C_{3},Z_{1},Z_{2},B,N$ or $W$.
\end{theorem}

The necessity of this theorem follows from Theorem 1 immediately.

It is known that the only 2-connected $\{K_{1,3},Z_3\}$-free
non-Hamiltonian graphs have 9 vertices (see
\cite{Faudree_Gould_Ryjacek_Schiermeyer}), hence for $n\geq 10$,
every 2-connected $\{K_{1,3},Z_3\}$-free graph is  also Hamiltonian.
This leads to the following

\begin{problem}
Is every 2-connected $\{K_{1,3},Z_3\}$-heavy graph on $n\geq 10$
vertices Hamiltonian?
\end{problem}

Instead of Theorems 3 and 4, we prove the following two stronger results.

\begin{theorem}
\label{thm-triple1}
If $G$ is a 2-connected $\{K_{1,3},N_{1,1,2},D\}$-heavy graph,
then $G$ is Hamiltonian (see Fig. 4).
\end{theorem}

\begin{theorem}
\label{thm-triple2}
If $G$ is a 2-connected  $\{K_{1,3},N_{1,1,2},H_{1,1}\}$-heavy
graph, then $G$ is Hamiltonian (see Fig. 4).
\end{theorem}

\begin{center}
\begin{picture}(180,160)

\thicklines

\put(0,0){\multiput(20,30)(50,0){2}{\multiput(0,0)(0,30){2}{\put(0,0){\circle*{6}}}
\put(0,0){\line(0,1){30}}}
\multiput(45,85)(0,30){3}{\put(0,0){\circle*{6}}}
\put(45,85){\line(0,1){60}} \put(20,60){\line(1,0){50}}
\put(20,60){\line(1,1){25}} \put(70,60){\line(-1,1){25}}
\put(35,10){$N_{1,1,2}$}}

\put(90,0){\multiput(20,35)(0,75){2}{\multiput(0,0)(50,0){2}{\put(0,0){\circle*{6}}}
\put(0,0){\line(1,0){50}}} \put(45,72.5){\circle*{6}}
\put(20,35){\line(2,3){50}} \put(70,35){\line(-2,3){50}}
\multiput(20,110)(50,0){2}{\put(0,0){\line(0,1){30}}
\put(0,30){\circle*{6}}} \put(35,10){$H_{1,1}$}}

\end{picture}

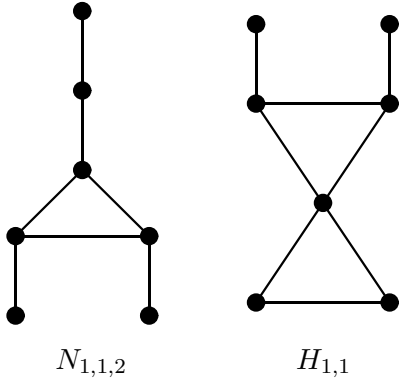
\small Fig. 4. Graphs $N_{1,1,2}$ and $H_{1,1}$.
\end{center}

Note that Brousek~\cite{Brousek} gave a complete characterization
of triples of connected graphs $K_{1,3},X,Y$ such that a graph $G$
being 2-connected and $\{K_{1,3},X,Y\}$-free implies $G$ is Hamiltonian.
Clearly, if $K_{1,3},S,T$ is a triple such that every 2-connected $\{K_{1,3},S,T\}$-heavy graph is Hamiltonian, then, for some triple
$K_{1,3},X,Y$ of \cite{Brousek}, $S$ and $T$ are induced subgraphs
of $X$ and $Y$, respectively (of course, the triples of
Theorems~\ref{thm-triple1} and \ref{thm-triple2} have this property).
We refer an interested reader to \cite{Brousek} for more details.

\section{Some preliminaries}

We first give some additional terminology and notation.

Let $G$ be a graph and $X$ be a subset of $V(G)$.
The subgraph of $G$ induced by the set $X$ is denoted $G[X]$.
We use $G-X$ to denote the subgraph induced by $V(G)\setminus X$.

Throughout this paper, $k$ and $\ell$ will always denote positive
integers, and we use $s$ and $t$ to denote integers which may be
nonpositive. For $s\leq t$, we use $[x_s,x_t]$ to denote the set
$\{x_s,x_{s+1},\ldots,x_t\}$. If $[x_s,x_t]$ is a subset of the
vertex set of a graph $G$, we use $G[x_s,x_t]$, instead of
$G[[x_s,x_t]]$, to denote the subgraph induced by $[x_s,x_t]$ in
$G$.

For a path $P$ and $x,y\in V(P)$, $P[x,y]$ denotes the subpath of
$P$ from $x$ to $y$. Similarly, for a cycle $C$ with a given
orientation and $x,y\in V(C)$, $\overrightarrow{C}[x,y]$ or
$\overleftarrow{C}[y,x]$ denotes the $(x,y)$-path on $C$ traversed
in the same or opposite direction with respect to the given
orientation of $C$.

Let $G$ be a graph and $x_1,x_2,y_1,y_2\in V(G)$ with $x_1\neq x_2$
and $y_1\neq y_2$. Then a subgraph $Q$ of $G$ such that $Q$ has
exactly 2 components, each of them being an
$(\{x_1,x_2\},\{y_1,y_2\})$-path, is called an
{\em $(x_1x_2,y_1y_2)$-disjoint path pair}, or briefly an
{\em $(x_1x_2,y_1y_2)$-pair} in $G$.

If $G$ is a graph on $n\geq 2$ vertices, $x\in V(G)$, and a graph
$G'$ is obtained from $G$ by adding a (new) vertex $y$ and a pair of
edges $yx$, $yz$, where $z$ is an arbitrary vertex of $G$, $z\neq
x$, we say that $G'$ is a {\em 1-extension of $G$ at $x$ to $y$}.
Similarly, if $x_1,x_2\in V(G)$, $x_1\neq x_2$, then the graph $G'$
obtained from $G$ by adding two (new) vertices $y_1,y_2$ and the
edges $y_1x_1$, $y_2x_2$ and $y_1y_2$ is called the {\em 2-extension
of $G$ at $(x_1,x_2)$ to $(y_1,y_2)$}.

Let $G$ be a graph and let $u,v,w\in V(G)$ be distinct vertices of $G$.
We say that $G$ is {\em $(u,v,w)$-composed} (or briefly {\em composed})
if $G$ has a spanning subgraph $D$ (called the {\em carrier} of $G$)
such that there is an ordering $v_{-k},\ldots,v_0,\ldots,v_\ell$
($k,\ell\geq 1$) of $V(D)$ (=$V(G)$) and a sequence of graphs
$D_1,\ldots,D_r$ ($r\geq 1$) such that
\begin{mathitem}
  \item[$(a)$] $u=v_{-k}$, $v=v_0$, $w=v_\ell$,
  \item[$(b)$] $D_1$ is a triangle with $V(D_1)=\{v_{-1},v_0,v_1\}$,
  \item[$(c)$] $V(D_i)=[v_{-k_i},v_{\ell_i}]$ for some $k_i$, $\ell_i$,
     $1\leq k_i\leq k$, $1\leq\ell_i\leq\ell$, and $D_{i+1}$ satisfies
     one of the following:
     \begin{mathitem}
       \item $D_{i+1}$ is a 1-extension of $D_i$ at $v_{-k_i}$ to
          $v_{-k_i-1}$ or at $v_{\ell_i}$ to $v_{\ell_i+1}$,
       \item $D_{i+1}$ is a 2-extension of $D_i$ at
          $(v_{-k_i},v_{\ell_i})$ to $(v_{-k_i-1},v_{\ell_i+1})$,
     \end{mathitem}
     $i=1,\ldots,r-1$,
  \item[$(d)$] $D_t=D$.
\end{mathitem}

The ordering $v_{-k},\ldots,v_0,\ldots,v_\ell$ will be called a {\em
canonical ordering} and the sequence $D_1,\ldots,D_r$ a {\em
canonical sequence} of $D$ (and also of $G$). Note that a composed
graph $G$ can have several carriers, canonical orderings and
canonical sequences. Clearly, a composed graph $G$ and any its
carrier $D$ are 2-connected, for any canonical ordering,
$P=v_{-k}\cdots v_0\cdots v_\ell$ is a Hamilton path in $D$ (called
a {\em canonical path}), and if $D_1,\ldots,D_r$ is a canonical
sequence, then any $D_i$ is $(v_{-k_i},v_0,v_{\ell_i})$-composed,
$i=1,\ldots,r$.

Now we give a lemma on composed graphs which will be needed in our
proofs.

\begin{lemma}
Let $G$ be a composed graph and let $D$ and
$v_{-k},\ldots,v_0,\ldots,v_\ell$ be a carrier and a canonical ordering
of $G$. Then
  \begin{mathitem}
    \item $D$ has a Hamilton $(v_0,v_{-k})$-path,
    \item for every $v_s\in V(G)\setminus\{v_{-k}\}$, $D$ has a spanning
      $(v_0v_\ell,v_sv_{-k})$-pair.
  \end{mathitem}
\end{lemma}

\begin{proof}
Let $D_1,\ldots,D_r$ be a canonical sequence and $Q$ the canonical path
of $D$ corresponding to the given ordering and, for every
$s\in [-k,\ell]\setminus\{0\}$, let $\hat{s}$, $1\leq\hat{s}\leq r$,
be the smallest integer for which $v_s\in V(D_{\hat{s}})$. Clearly,
$d_{D_{\hat{s}}}(v_s)=2$.

Now we prove $(i)$ by induction on $|V(D)|$.
If $|V(D)|=3$, the assertion is trivially true.
Suppose now that $|V(D)|\geq 4$ and the assertion is true for
every graph with at most $|V(D)|-1$ vertices. By the
definition of a carrier, we have the following two cases.

\begin{case}
$V(D_{r-1})=[v_{-k+1},v_\ell]$ and $D$ is a 1-extension of $D_{r-1}$
at $v_{-k+1}$ to $v_{-k}$.
\end{case}

By the induction hypothesis, $D_{r-1}$ has a Hamilton
$(v_0,v_{-k+1})$-path $P'$. Then $P=v_0P'v_{-k+1}v_{-k}$ is a
Hamilton $(v_0,v_{-k})$-path in $D$.

\begin{case}
$V(D_{r-1})=[v_{-k},v_{\ell-1}]$ and $D$ is a 1-extension of
$D_{r-1}$ at $v_{\ell-1}$ to $v_\ell$, or
$V(D_{r-1})=[v_{-k+1},v_{\ell-1}]$ and $D$ is a 2-extension of
$D_{r-1}$ at $(v_{-k+1},v_{\ell-1})$ to $(v_{-k},v_\ell )$.
\end{case}

In this case, $v_\ell$ has a neighbor $v_s$ other than $v_{\ell-1}$,
where $s\in [-k,\ell-2]$.

\begin{subcase}
$s\in [-k,-2]$.
\end{subcase}

In this case $s+1\in [-k+1,-1]$. Consider the graph
$D'=D_{\widehat{s+1}}$. Let $V(D')=[v_{s+1},v_t]$, where $t>0$. By
the induction hypothesis, there exists a Hamilton $(v_0,v_t)$-path
$P'$ of $D'$. Then the path
$P=P'Q[v_t,v_{\ell}]v_{\ell}v_sQ[v_s,v_{-k}]$ is a Hamilton
$(v_0,v_{-k})$-path of $D$.

\begin{subcase}
$s=-1$.
\end{subcase}

In this case, the path
$P=Q[v_0,v_{\ell}]v_{\ell}v_{-1}Q[v_{-1},v_{-k}]$ is a Hamilton
$(v_0,v_{-k})$-path of $D$.

\begin{subcase}
$s\in [0,l-2]$.
\end{subcase}

In this case $s+1\in [1,\ell-1]$. Consider the graph
$D'=D_{\widehat{s+1}}$. Let $V(D')=[v_t,v_{s+1}]$, where $t<0$ and
$d_{D'}(v_{s+1})=2$. By the induction hypothesis, there exists a
Hamilton $(v_0,v_t)$-path $P'$ of $D'$, and the edge $v_sv_{s+1}$ is
in $E(P')$ by the fact $d_{D'}(v_{s+1})=2$. Thus the path
$P=P'-v_sv_{s+1}\cup Q[v_{s+1},v_l]v_lv_s\cup Q[v_t,v_{-k}]$ is a
Hamilton $(v_0,v_{-k})$-path of $G$.

So the proof of $(i)$ is complete.

Now we prove $(ii)$. We distinguish the following three cases.

\setcounter{case}{0}
\begin{case}
$s\in [-k+1,0]$.
\end{case}

In this case, $s-1\in [-k,-1]$. Consider the graph
$D'=D_{\widehat{s-1}}$. Let $V(D')=[v_{s-1},v_t]$, where $t>0$ and
$d_{D'}(v_{s-1})=2$. By $(i)$, there exists a Hamilton
$(v_0,v_t)$-path $P'$ of $D'$ and $v_{s-1}v_s\in E(P')$. Thus
$R'=P'-v_{s-1}v_s$ is a spanning $(v_0v_t,v_sv_{s-1})$-pair of $D'$,
and $R=R'\cup Q[v_t,v_l]\cup Q[v_{s-1},v_{-k}]$ is a spanning
$(v_0v_{\ell},v_sv_{-k})$-pair of $D$.

\begin{case}
$s=1$.
\end{case}

In this case, $R=Q[v_0,v_{-k}]\cup Q[v_1,v_{\ell}]$ is a spanning
$(v_0v_{\ell},v_1v_{-k})$-pair of $D$.

\begin{case}
$s\in [2,\ell]$.
\end{case}

In this case, $s-1\in [1,l-1]$. Consider the graph
$D'=D_{\widehat{s-1}}$. Let $V(D')=[v_t,v_{s-1}]$, where $t<0$. By
$(i)$, there exists a Hamilton $(v_0,v_t)$-path $P'$ of $G'$. Thus
$P_1=P'Q[v_t,v_{-k}]$ and $P_2=Q[v_s,v_{\ell}]$ form a spanning
$(v_0v_{\ell},v_sv_{-k})$-pair of $D$.

The proof is complete.
\end{proof}

\medskip

Let $G$ be a graph on $n$ vertices and $k\geq 3$ an integer. A
sequence of vertices $C=v_1v_2\cdots v_kv_1$ such that for all $i\in
[1,k]$ either $v_iv_{i+1}\in E(G)$ or $d(v_i)+d(v_{i+1})\geq n$
(indices modulo $k$) is called an \emph{Ore-cycle} {or briefly, \em
$o$-cycle} of $G$. The {\em deficit} of an $o$-cycle $C$ is the
integer $\de(C)=|\{i\in[1,k]: v_iv_{i+1}\notin E(G)\}|$. Thus, a
cycle is an $o$-cycle of deficit 0. Similarly we define an {\em
$o$-path} of $G$.

Now, we prove the following lemma on $o$-cycles.

\begin{lemma}
\label{lemma-$o$-cycle}
Let $G$ be a graph and let $C'$ be an $o$-cycle in $G$. Then there is
a cycle $C$ in $G$ such that $V(C')\subset V(C)$.
\end{lemma}

\begin{proof}
Let $C_1$ be an $o$-cycle in $G$ such that $V(C')\subset V(C_1)$ and
$\de(C_1)$ is smallest possible, and suppose, to the contrary,
that $\de(C_1)\geq 1$. Without loss of generality suppose that
$C_1=v_1v_2\ldots v_kv_1$, where $v_1v_k\notin E(G)$ and
$d(v_1)+d(v_k)\geq n$. We use $P$ to denote the $o$-path
$P=v_1v_2\cdots v_k$.

If $v_1$ and $v_k$ have a common neighbor $x\in V(G)\setminus V(P)$,
then $C_2=v_1Pv_kxv_1$ is an $o$-cycle in $G$ with
$V(C')\subset V(C_2)$ and $\de(C_2)<\de(C_1)$, a contradiction.
Hence $N_{G-P}(v_1)\cap N_{G-P}(v_k)=\emptyset$. Then we have
$d_P(v_1)+d_P(v_k)\geq |V(P)|$ since $d(v_1)+d(v_k)\geq n$.
Thus, there exists $i\in[2,k-1]$ such that $v_i\in N_P(v_1)$ and
$v_{i-1}\in N_P(v_k)$, and then again
$C_2=v_1P[v_1,v_{i-1}]v_{i-1}v_kP[v_k,v_i]v_iv_1$ is an $o$-cycle
with $V(C')\subset V(C_2)$ and $\de(C_2)<\de(C_1)$, a contradiction.
\end{proof}

\medskip

Note that Lemma~\ref{lemma-$o$-cycle} immediately implies that if
$P$ is an $(x,y)$-path or an $o$-path in $G$ with $|V(P)|$ larger than
the length of a longest cycle in $G$, then $xy\notin E(G)$ and
$d(x)+d(y)<n$.

In the following, we denote
$\overline{E}(G)=\{uv: uv\in E(G)$ or $d(u)+d(v)\geq n\}$.

\medskip

Let $C$ be a cycle in $G$, $x,x_1,x_2\in V(C)$ three distinct vertices,
and set $X=V(Q)$, where $Q$ is the $(x_1,x_2)$-path on $C$ containing $x$.
We say that the pair of vertices $(x_1,x_2)$ is {\em $x$-good on $C$},
if for some $j\in\{1,2\}$ there is a vertex
$x'\in X\setminus\{x_j\}$ such that
\begin{mathitem}
\item[(1)] there is an $(x,x_{3-j})$-path $P$ such that
     $V(P)=X\setminus\{x_j\}$,
\item[(2)] there is an $(xx_{3-j},x'x_j)$-pair $D$ such
     that $V(D)=X$,
\item[(3)] $d(x_j)+d(x')\geq n$.
\end{mathitem}

\begin{lemma}
Let $G$ be a graph, and $C$ be a cycle of $G$ with a given
orientation. Let $x,y\in V(C)$ and let $R$ be an $(x,y)$-path in $G$
which is internally disjoint with $C$. If there are vertices
$x_1,x_2,y_1,y_2\in V(C)\setminus\{x,y\}$ such that
\begin{mathitem}
\item $x_2,x,x_1,y_1,y,y_2$ appear in this order along $\overrightarrow{C}$
     (possibly $x_1=y_1$ or $x_2=y_2$),
\item $(x_1,x_2)$ is $x$-good on $C$,
\item $(y_1,y_2)$ is $y$-good {on} $C$,
\end{mathitem}
then there is a cycle $C'$ in $G$ such that
$V(C)\cup V(R)\subset V(C')$.
\end{lemma}

\begin{proof}
Assume the opposite. Let $P_1$ and $D_1$ be the path and disjoint
path pair associated with $x$, and $P_2$ and $D_2$ associated with
$y$; and let $Q_1=\overrightarrow{C}[x_1,y_1]$ and
$Q_2=\overleftarrow{C}[x_2,y_2]$.

By the definition of an $x$-good pair, without loss of generality,
we can assume that $P_1$ is an $(x,x_1)$-path, $D_1$ is an
$(xx_1,x'x_2)$-pair, and $d(x_2)+d(x')\geq n$.

\setcounter{case}{0}
\begin{case}
$P_2$ is a $(y,y_1)$-path, $D_2$ is a $(yy_1,y'y_2)$-pair, and
$d(y_2)+d(y')\geq n$.
\end{case}

In this case the path $P=Q_2\cup D_2\cup R\cup P_1\cup Q_1$ is an
$(x_2,y')$-path which contains all the vertices in $V(C)\cup V(R)$,
and $P'=Q_2\cup D_1\cup R\cup P_2\cup Q_1$ is an $(x',y_2)$-path
which contains all the vertices in $V(C)\cup V(R)$. Thus, by Lemma 2,
$d(x_2)+d(y')<n$ and $d(x')+d(y_2)<n$, a contradiction to
$d(x_2)+d(x')\geq n$ and $d(y_2)+d(y')\geq n$.

\begin{case}
$P_2$ is a $(y,y_2)$-path, $D_2$ is a $(yy_2,y'y_1)$-pair, and
$d(y_1)+d(y')\geq n$.
\end{case}

\begin{subcase}
The $(xx_1,x'x_2)$-pair $D_1$ is formed by an $(x,x_2)$-path and an
$(x_1,x')$-path.
\end{subcase}

In this case, the path $P=Q_2\cup P_2\cup R\cup P_1\cup Q_1$ is an
$({x_2,y_1})$-path which contains all the vertices in
$V(C)\cup V(R)$, and the path $P'=D_1\cup Q_1\cup Q_2\cup R\cup D_2$
is an $(x',y')$-path which contains all the vertices in $V(C)\cup
V(R)$. By Lemma 2, $d(x_2)+d(y_1)<n$ and $d(x')+d(y')<n$, a
contradiction.

\begin{subcase}
The $(xx_1,x'x_2)$-pair $D_1$ is formed by an $(x,x')$-path and an
$(x_1,x_2)$-path.
\end{subcase}

\begin{subsubcase}
The $(yy_2,y'y_1)$-pair $D_2$ is formed by an $(y,y_1)$-path and an
$(y_2,y')$-path.
\end{subsubcase}

This case can be proved similarly as in Case 2.1.

\begin{subsubcase}
The $(yy_2,y'y_1)$-pair $D_2$ is formed by an $(y,y')$-path and
an $(y_1,y_2)$-path.
\end{subsubcase}

In this case, the path $P=Q_2\cup D_2\cup R\cup P_1\cup Q_1$ is an
$(x_2,y')$-path containing all vertices in $V(C)\cup V(R)$,
and the path $P'=Q_2\cup D_1\cup R\cup P_2\cup Q_1$ is an
$(x',y_1)$-path containing all vertices in $V(C)\cup V(R)$.
By Lemma 2, $d(x_2)+d(y')<n$ and $d(x')+d(y_1)<n$, a contradiction.

The proof is complete.
\end{proof}

\section{Proof of Theorem 8}

Let $C$ be a longest cycle of $G$ with a given orientation,
set $n=|V(G)|$ and $c=|V(C)|$, and assume
that $G$ is not Hamiltonian, i.e. $c<n$. Then $V(G)\backslash V(C)\neq
\emptyset$. Since $G$ is 2-connected, there exists a
$(u_0,v_0)$-path with length at least 2 which is internally disjoint
from $C$, where $u_0,v_0\in V(C)$. Let $R=z_0z_1z_2\cdots z_{r+1}$,
where $z_0=u_0$ and $z_{r+1}=v_0$, be such a path,
and choose $R$ as short as possible.
Let $r_1$ and $r_2$ denote the number of interior vertices in
$\overrightarrow{C}[u_0,v_0]$ and $\overrightarrow{C}[v_0,u_0]$,
respectively (note that clearly $r_1+r_2+2=c$). We denote the
vertices of $C$ by $\overrightarrow{C}=u_0u_1u_2\cdots
u_{r_1}v_0u_{-r_2}u_{-r_2+1}\cdots u_{-1}u_0$ or
$\overleftarrow{C}=v_0v_1v_2\cdots
v_{r_1}u_0v_{-r_2}v_{-r_2+1}\cdots v_{-1}v_0$, where
$u_{\ell}=v_{r_1+1-\ell}$ and $u_{-k}=v_{-r_2-1+k}$
{(see Fig. 5)}. Let $H$ be the component of $G-C$
which contains the vertices in $[z_1,z_r]$.

\begin{center}
\begin{picture}(200,200)

\thicklines

\put(100,100){\circle{160}} \put(20,100){\line(1,0){160}}
\multiput(20,100)(20,0){4}{\put(0,0){\circle*{4}}}
\multiput(120,100)(20,0){4}{\put(0,0){\circle*{4}}}

\put(100,100){\put(77.3,20.7){\circle*{4}}
\put(69.3,40){\circle*{4}} \put(56.6,56.6){\circle*{4}}
\put(-77.3,20.7){\circle*{4}} \put(-69.3,40){\circle*{4}}
\put(-56.6,56.6){\circle*{4}} \put(77.3,-20.7){\circle*{4}}
\put(69.3,-40){\circle*{4}} \put(56.6,-56.6){\circle*{4}}
\put(-77.3,-20.7){\circle*{4}} \put(-69.3,-40){\circle*{4}}
\put(-56.6,-56.6){\circle*{4}} \put(82.3,20.7){$u_1$}
\put(74.3,40){$u_2$} \put(61.5,56.6){$u_3$}
\put(-96.3,40){$u_{r_1-1}$} \put(-94.3,20.7){$u_{r_1}$}
\put(82.3,-20.7){$u_{-1}$} \put(74.3,-40){$u_{-2}$}
\put(61.5,-56.6){$u_{-3}$} \put(-101.3,-40){$u_{-r_2+1}$}
\put(-99.3,-20.7){$u_{-r_2}$}

\put(-72.3,20.7){$v_1$} \put(-64.3,40){$v_2$}
\put(-51.6,56.6){$v_3$} \put(41.3,40){$v_{r_1-1}$}
\put(59.3,20.7){$v_{r_1}$} \put(-72.3,-20.7){$v_{-1}$}
\put(-64.3,-40){$v_{-2}$} \put(-51.6,-56.6){$v_{-3}$}
\put(36.3,-40){$v_{-r_2+1}$} \put(54.3,-20.7){$v_{-r_2}$}}

\put(8,100){$v_0$} \put(20,105){$z_{r+1}$} \put(40,105){$z_r$}
\put(135,105){$z_2$} \put(155,105){$z_1$} \put(170,105){$z_0$}
\put(185,100){$u_0$}
\end{picture}

\small Fig. 5. $C\cup R$, the subgraph of $G$.
\end{center}

\begin{claim}
Let $x\in V(H)$ and $y\in \{u_1,u_{-1},v_1,v_{-1}\}$. Then $xy\notin
\overline{E}(G)$.
\end{claim}

\begin{proof}
Without {loss} of generality, we
{assume} $y=u_1$. Let $P'$ be an $(x,z_1)$-path in
$H$. Then $P=P'z_1u_0\overleftarrow{C}[u_0,u_1]$ is an $(x,y)$-path
which contains all the vertices in $V(C)\cup V(P')$. By Lemma 2, we
have $xy\notin \overline{E}(G)$.
\end{proof}

\begin{claim}
$u_1u_{-1}\in \overline{E}(G)$, $v_1v_{-1}\in \overline{E}(G)$.
\end{claim}

\begin{proof}
If $u_1u_{-1}\notin E(G)$, by Claim 1, the graph induced by
$\{u_0,z_1,u_1,u_{-1}\}$ is a claw, where $d(z_1)+d(u_{\pm 1})<n$.
Since $G$ is a claw-heavy graph, we have that $d(u_1)+d(u_{-1})\geq
n$.

The second assertion can be proved similarly.
\end{proof}

\begin{claim}
$u_1v_{-1}\notin \overline{E}(G)$, $u_{-1}v_1\notin
\overline{E}(G)$, $u_0v_{\pm 1}\notin \overline{E}(G),u_{\pm
1}v_0\notin \overline{E}(G)$.
\end{claim}

\begin{proof}
Since $P=\overrightarrow{C}[u_1,v_0]R\overleftarrow{C}[u_0,v_{-1}]$
is a $(u_1,v_{-1})$-path which contains all the vertices in
$V(C)\cup V(R)$, we have $u_1v_{-1}\notin \overline{E}(G)$ by Lemma
2.

If $u_0v_1\in \overline{E}(G)$, then
$Pc=\overrightarrow{C}[u_1,v_1]v_1u_0R\overrightarrow{C}[v_0,u_{-1}]u_{-1}u_1$
is an $o$-cycle which contains all the vertices of $V(C)\cup V(R)$.
By Lemma 2, there exists a cycle which contains all the vertices in
$V(C)\cup V(R)$, a contradiction.

The other {assertions} can be proved similarly.
\end{proof}

\begin{claim}
Either $u_1u_{-1}\in E(G)$ or $v_1v_{-1}\in E(G)$.
\end{claim}

\begin{proof}
Assume the opposite. By Claim 2 we have $d(u_1)+d(u_{-1})\geq n$ and
$d(v_1)+d(v_{-1})\geq n$. By Claim 3, we have $d(u_1)+d(v_{-1})<n$
and $d(u_{-1})+d(v_1)<n$, a contradiction.
\end{proof}

Now, we distinguish two cases.

\setcounter{case}{0}
\begin{case}
$r\geq 2$, or $r=1$ and $u_0v_0\notin E(G)$.
\end{case}

By Claim 4, without {loss} of generality, we assume that
$u_1u_{-1}\in E(G)$. Thus $G[u_{-1},u_1]$ is
$(u_{-1},u_0,u_1)$-composed.

\begin{claim}
$z_2u_0\notin \overline{E}(G)$.
\end{claim}

\begin{proof}
By the choice of the path $R$, we have $z_2u_0\notin E(G)$. Now we
prove that $d(z_2)+d(u_0)<n$.

\begin{subclaim}
Every neighbor of $u_0$ is in $V(C)\cup V(H)$; every neighbor of
$z_2$ is in $V(C)\cup V(H)$.
\end{subclaim}

\begin{proof}
Assume the opposite. Let $z'\in V(H')$ be a neighbor of $u_0$ where
$H'$ is a component of $G-C$ other than $H$. Then we have
$z'z_1\notin E(G)$ and $N_{G-C}(z')\cap N_{G-C}(z_1)=\emptyset$.

By Claim 1, we have $u_1z_1\notin \overline{E}(G)$, and similarly
$u_1z'\notin \overline{E}(G)$. Thus the graph induced by
$\{u_0,u_1,z_1,z'\}$ is a claw, where $d(u_1)+d(z_1)<n$ and
$d(u_1)+d(z')<n$. Then we have $d(z_1)+d(z')\geq n$.

Since $N_{G-C}(z_1)\cap N_{G-C}(z')=\emptyset$, there
{exist} two vertices $x_1,x_2\in V(C)$ such that
{$x_1x_2\in E(\overrightarrow{C})$} and
$z_1x_1,z'x_2\in E(G)$. Thus
$P=z_1x_1\overleftarrow{C}[x_1,x_2]x_2z'$ is a $(z_1,z')$-path which
contains all the vertices in $V(C)\cup\{z_1,z'\}$. By Lemma 2, there
exists a cycle which contains all the vertices in
$V(C)\cup\{z_1,z'\}$, a contradiction.

If $z_2=v_0$, the second assertion can be proved similarly; and if
$z_2\neq v_0$, {the assertion} is obvious.
\end{proof}

Let $h=|V(H)|$ and $k=|N_H(u_0)|$. Then we have
$d_H(z_2)+d_H(u_0)\leq h+k$. Since $z_1\in N_H(u_0)$, we have $k\geq
1$. Let $N_H(u_0)=\{y_1,y_2,\ldots,y_k\}$, where $y_1=z_1$.

\begin{subclaim}
$y_iy_j\in \overline{E}(G)$ for all $1\leq i<j\leq k$.
\end{subclaim}

\begin{proof}
If $y_iy_j\notin E(G)$, then by Claim 1, the graph induced by
$\{u_0,u_1,y_i,y_j\}$ is a claw, where $d(y_i)+d(u_1)<n$ and
$d(y_j)+d(u_1)<n$. Thus we have $d(y_i)+d(y_j)\geq n$.
\end{proof}

Now, let $Q$ be the $o$-path $Q=z_2y_1y_2\cdots y_ku_0$. It is clear
that $R[z_2,v_0]$ and $Q$ are internally disjoint, and $Q$ contains
at least $k$ vertices in $V(H)$. In the following, we use $C'$ to
denote the cycle $\overrightarrow{C}[u_1,u_{-1}]u_{-1}u_1$ if
$z_2\neq v_0$, and to denote the $o$-cycle
$\overrightarrow{C}[u_1,v_1]v_1v_{-1}\overrightarrow{C}[v_{-1},u_{-1}]u_{-1}u_1$ if $z_2=v_0$.

By Claims 1 and 3, we have $z_2v_{r_1}\notin E(G)$, where
$v_{r_1}=u_1$. Let $v_{\ell}$ be the last vertex in
$\overleftarrow{C}[v_1,u_1]$ such that $z_2v_{\ell}\in E(G)$. If
there are no neighbors of $z_2$ in $\overleftarrow{C}[v_1,u_1]$,
then let $v_{\ell}=v_0$.

\begin{subclaim}
For every vertex $v_{{\ell}'}\in N_{[v_1,v_{r_1}]}(z_2)\cup\{v_0\}$,
$u_0v_{{\ell}'+1}\notin E(G)$.
\end{subclaim}

\begin{proof}
By Claim 3, we have $u_0v_1\notin E(G)$.

If $z_2v_{{\ell}'}\in E(G)$ and $u_0v_{{\ell}'+1}\in E(G)$, then
$C''=\overrightarrow{C'}[v_{{\ell}'},v_{{\ell}'+1}]v_{{\ell}'+1}u_0Qz_2v_{{\ell}'}$
is an $o$-cycle which contains all the vertices in $V(C)\cup V(Q)$,
a contradiction.
\end{proof}

\begin{subclaim}
$r_1-\ell\geq k+1$, and for every vertex
$v_{{\ell}'}\in[v_{{\ell}+1},v_{{\ell}+k}]$, $u_0v_{{\ell}'}\notin
E(G)$.
\end{subclaim}

\begin{proof}
Assume the opposite. Let $v_{{\ell}'}$ be the first vertex in
$[v_{{\ell}+1},v_{r_1}]$ such that $u_0v_{{\ell}'}\in E(G)$, and
${\ell}'-{\ell}<k+1$.

If $v_{\ell}=v_0$, then
$C''=\overrightarrow{C}[v_0,u_{-1}]u_{-1}u_1\overrightarrow{C}[u_1,v_{{\ell}'}]v_{{\ell}'}u_0QR[z_2,v_0]$
is an $o$-cycle which contains all the vertices in
$V(C)\backslash[v_1,v_{{\ell}'-1}]\cup V(Q)$, and $|V(C'')|>c$, a
contradiction.

Thus, we assume that $v_{\ell}\neq v_0$, {and} $z_2v_{\ell}\in
E(G)$. {Then}
$C''=\overrightarrow{C'}[v_{\ell},v_{{\ell}'}]v_{{\ell}'}u_0Qz_2v_{\ell}$
is an $o$-cycle which contains all the vertices in
$V(C)\backslash[v_{\ell+1},v_{{\ell}'-1}]\cup V(Q)$, and
$|V(C'')|>c$, a contradiction.

Thus we have ${\ell}'-\ell\geq k+1$. Note that $u_0v_{r_1}\in E(G)$,
we have $r_1-\ell\geq k+1$.
\end{proof}

Let $d_1=|N_{[v_1,v_{r_1}]}(z_2)\cup\{v_0\}|$,
$d_2=|N_{[v_{-r_2},v_{-1}]}(z_2)\cup\{v_0\}|$,
$d'_1=|N_{[v_1,v_{r_1}]}(u_0)|$ and
$d'_2=|N_{[v_{-r_2},v_{-1}]}(u_0)|$. Then $d_C(z_2)\leq d_1+d_2-1$
and $d_C(u_0)\leq d'_1+d'_2+1$.

By Claims 5.3 and 5.4, we have $d'_1\leq r_1-d_1-k+1$, and
similarly, $d'_2\leq r_2-d_2-k+1$. Thus $d_C(z_2)+d_C(u_0)\leq
r_1+r_2-2k+2=c-2k$. Note that $d_H(z_2)+d_H(u_0)\leq h+k$. By Claim
5.1, we have $d(z_2)+d(u_0)\leq c+h-k<n$.
\end{proof}

Recall that $G[u_{-1},u_1]$ is $(u_{-1},u_0,u_1)$-composed. Now we
prove the following claims.

\begin{claim}
If $G[u_{-k},u_{\ell}]$ is $(u_{-k},u_0,u_{\ell})$-composed with
canonical ordering $u_{-k},u_{-k+1},\ldots,u_{\ell}$, then $k\leq
r_2-2$ and $\ell\leq r_1-2$.
\end{claim}

\begin{proof}
Let $D_1,D_2,\ldots,D_r$ be a canonical sequence of
$G[u_{-k},u_{\ell}]$ corresponding to the canonical ordering
$u_{-k},u_{-k+1},\ldots,u_{\ell}$. Suppose that $k>r_2-2$. Consider
the the graph $D'=D_{\widehat{-r_2+1}}$, where $\widehat{-r_2+1}$ be
the smallest integer such that $u_{-r_2+1}\in
V(D_{\widehat{-r_2+1}})$. Let $V(D')=[u_{-r_2+1},u_{{\ell}'}]$. By
Lemma 1, there exists a $(u_0,u_{{\ell}'})$-path $P$ such that
$V(P)=[u_{-r_2+1},u_{{\ell}'}]$. Then
$C'=v_{-1}v_0RP\overrightarrow{C}[u_{{\ell}'},v_1]v_1v_{-1}$ is an
$o$-cycle which contains all the vertices in $V(C)\cup V(R)$, a
contradiction.
\end{proof}

\begin{claim}
If $G[u_{-k},u_{\ell}]$ is $(u_{-k},u_0,u_{\ell})$-composed with
canonical ordering $u_{-k},u_{-k+1},\ldots,u_{\ell}$, where $k\leq
r_2-2$ and $l\leq r_1-2$, and any two nonadjacent vertices in
$[u_{-k-1},u_{\ell+1}]$ have degree sum less than $n$, then one of
the following is true:\\
(1) $G[u_{-k-1},u_{\ell}]$ is $(u_{-k-1},u_0,u_{\ell})$-composed
with canonical ordering $u_{-k-1},u_{-k},\ldots,u_{\ell}$,\\
(2) $G[u_{-k},u_{\ell+1}]$ is $(u_{-k},u_0,u_{\ell+1})$-composed
with canonical ordering $u_{-k},u_{-k+1},\ldots,u_{\ell+1}$,\\
(3) $G[u_{-k-1},u_{\ell+1}]$ is $(u_{-k-1},u_0,u_{\ell+1})$-composed
with canonical ordering $u_{-k-1},u_{-k},\ldots,u_{\ell+1}$.
\end{claim}

\begin{proof}
Assume the opposite, which implies that for every vertex
$u_s\in[u_{-k+1},u_{\ell}]$, $u_{-k-1}u_s\notin E(G)$, and for every
vertex $u_s\in[u_{-k},u_{\ell-1}]$, we have $u_{\ell+1}u_s\notin E(G)$
and $u_{-k-1}u_{\ell+1}\notin E(G)$.

\begin{subclaim}
Let $z\in\{z_1,z_2\}$ and
$u_s\in[u_{-k-1},u_{\ell+1}]\backslash\{u_0\}$. Then $zu_s\notin
\overline{E}(G)$.
\end{subclaim}

\begin{proof}
Without {loss} of generality, we {assume} that $s>0$. If $s=1$, the
assertion is true by Claims 1 and 3. So we {assume} that
$s\in [2,\ell+1]$ and $s-1\in[1,\ell]$. By the definition of a
composed graph, there
exists $t\in [-k,-1]$ such that $G[u_t,u_{s-1}]$ is
$(u_t,u_0,u_{s-1})$-composed. By Lemma 1, there exists a
$(u_0,u_t)$-path $P'$ such that $V(P')=[u_t,u_{s-1}]$.

If $z\neq v_0$, then $P=R[z,u_0]P'\overleftarrow{C}[u_t,u_s]$ is a
$(z,u_s)$-path which contains all the vertices in $V(C)\cup\{z\}$.
By Lemma 2, we have $zu_s\notin \overline{E}(G)$.

If $z=v_0$ and $v_0u_s\in \overline{E}(G)$, then
$C'=RP'\overleftarrow{C}[u_t,v_{-1}]v_{-1}v_1\overleftarrow{C}[v_1,u_s]u_sv_0$
is an $o$-cycle which contains all the vertices in $V(C)\cup V(R)$,
a contradiction.
\end{proof}

Let $G'=G[[u_{-k-1},u_\ell]\cup\{z_1,z_2\}]$ and
$G''=G[[u_{-k-1},u_{\ell+1}]\cup\{z_1,z_2\}]$.

\begin{subclaim}
$G''$, and then $G'$, is $\{K_{1,3},N_{1,1,2}\}$-free.
\end{subclaim}

\begin{proof}
By Claims 5 and 7.1, and the condition that any two nonadjacent
vertices in $[u_{-k-1},u_{\ell+1}]$ have degree sum less than $n$, we
have {that} any two nonadjacent vertices in $G''$
have degree sum less than $n$. Since $G$ (and then $G''$) is
$\{K_{1,3},N_{1,1,2}\}$-heavy, we have that $G''$ is
$\{K_{1,3},N_{1,1,2}\}$-free.
\end{proof}

\begin{subclaim}
$N_{G'}(u_0)\backslash\{z_1\}$ is a clique.
\end{subclaim}

\begin{proof}
If there are two vertices $x,x'\in N_{G'}(u_0)\backslash\{z_1\}$
such that $xx'\notin E(G')$, then the graph induced by
$\{u_0,z_1,x,x'\}$ is a claw, a contradiction.
\end{proof}

Now, we define $N_i=\{x\in V(G'): d_{G'}(x,u_{-k-1})=i\}$. Then we
have $N_0=\{u_{-k-1}\}$, $N_1=\{u_{-k}\}$ and
$N_2=N_{G'}(u_{-k})\backslash\{u_{-k-1}\}$.

By the definition of a composed graph, we have $|N_2|\geq 2$. If there are
two vertices $x,x'\in N_2$ such that $xx'\notin E(G')$, then the
graph induced by $\{u_{-k},u_{-k-1},x,x'\}$ is a claw, a
contradiction. Thus, $N_2$ is a clique.

We assume $u_0\in N_j$, where $j\geq 2$. Then $z_1\in N_{j+1}$ and
$z_2\in N_{j+2}$.

If $|N_i|=1$ for some $i\in[2,j-1]$, say, $N_i=\{x\}$, then $x$ is a
cut vertex of the graph $G[u_{-k},u_l]$. By the definition of a
composed graph, $G[u_{-k},u_l]$ is 2-connected. This implies
$|N_i|\geq 2$ for every $i\in[2,j-1]$.

\begin{subclaim}
For $i\in [1,j]$, $N_i$ is a clique.
\end{subclaim}

\begin{proof}
We prove this claim by induction on $i$. For $i=1,2$, the claim is
true by the analysis above. So we {assume} that
$3\leq i\leq j$, and we have that $N_{i-3},N_{i-2},N_{i-1},N_{i+1}$ and
$N_{i+2}$ is nonempty, and $|N_{i-1}|\geq 2$.

First we choose a vertex $x\in N_i$ which has a neighbor $y\in
N_{i+1}$ {such that it} has a neighbor $z\in
N_{i+2}$. We prove that for every $x'\in N_i$, $xx'\in E(G)$. We
{assume} that $xx'\notin E(G)$.

If $x'y\in E(G)$, then the graph induced by $\{y,x,x',z\}$ is a
claw, a contradiction. Thus, we have $x'y\notin E(G)$. If $x$ and
$x'$ have a common neighbor in $N_{i-1}$, denote
{it} by $w$, then let $v$ be a neighbor of $w$ in
$N_{i-2}$, and the graph induced by $\{w,v,x,x'\}$ is a claw, a
contradiction. Thus we have that $x$ and $x'$ have no common
neighbors in $N_{i-1}$.

Let $w$ be a neighbor of $x$ in $N_{i-1}$ and $w'$ be a neighbor of
$x'$ in $N_{i-1}$. Then $xw',x'w\notin E(G)$. Let $v$ be a neighbor
of $w$ in $N_{i-2}$ and $u$ be a neighbor of $v$ in $N_{i-3}$. If
$w'v\notin E(G)$, then the graph induced by $\{w,v,w',x\}$ is a
claw, a contradiction. Thus we have $w'v\in E(G)$, and then the
graph induced by $\{v,u,w',x',w,x,y\}$ is an $N_{1,1,2}$, a
contradiction.

Thus we have $xx'\in E(G)$ for every $x'\in N_i$.

Now, let $x'$ and $x''$ be two vertices in $N_i$ other than $x$
such that $x'x''\notin E(G)$. We have $xx',xx''\in E(G)$.

If $x'y\in E(G)$, then similarly to the case of $x$, we have
$x'x''\in E(G)$, a contradiction. Thus we have $x'y\notin E(G)$.
Similarly, $x''y\notin E(G)$. Then the graph induced by
$\{x,x',x'',y\}$ is a claw, a contradiction.

Thus, $N_i$ is a clique.
\end{proof}

If there exists some vertex $y\in N_{j+1}$ other than $z_1$, then we
have $yu_0\notin E(G)$ by Claim 7.3. Let $x$ be a neighbor of $y$ in
$N_j$, $w$ be a neighbor of $u_0$ in $N_{j-1}$ and $v$ be a neighbor
of $w$ in $N_{j-2}$. Then $xu_0\in E(G)$ by Claim 7.4 and $xw\in
E(G)$ by Claim 7.3. Thus the graph induced by
$\{w,v,x,y,u_0,z_1,z_2\}$ is an $N_{1,1,2}$, a contradiction. So we
assume that all vertices in $[u_{-k},u_{\ell}]$ are in
$\bigcup_{i=1}^jN_i$.

If $u_{\ell}\in N_j$, then let $w$ be a neighbor of $u_0$ in
$N_{j-1}$ and $v$ be a neighbor of $w$ in $N_{j-2}$. Then the graph
induced by $\{w,v,u_0,z_1,u_{\ell},u_{{\ell}+1}\}$ is an
$N_{1,1,2}$, a contradiction. Thus we have that $u_{\ell}\notin N_j$
and then $j\geq 3$.

Let $u_{\ell}\in N_i$, where $i\in[2,j-1]$. If $u_{\ell}$ has a
neighbor in $N_{i+1}$, then let $y$ be a neighbor of $u_{\ell}$ in
$N_{i+1}$, and $w$ be a neighbor of $u_{\ell}$ in $N_{i-1}$. Then
the graph induced by $\{u_{\ell},w,y,u_{\ell+1}\}$ is a claw, a
contradiction. So we have that $u_{\ell}$ has no neighbors in
$N_{i+1}$.

Let $x\in N_i$ be a vertex other than $u_{\ell}$ which has a
neighbor $y$ in $N_{i+1}$ such that it has a neighbor $z$ in
$N_{i+2}$. Let $w$ be a neighbor of $x$ in $N_{i-1}$, and $v$ be a
neighbor of $w$ in $N_{i-2}$. If $u_{\ell}w\notin E(G)$, then the
graph induced by $\{x,w,u_{\ell},y\}$ is a claw, a contradiction. So
we have that $u_{\ell}w\in E(G)$. Then the graph induced by
$\{w,v,u_{\ell},u_{\ell+1},x,y,z\}$ is an $N_{1,1,2}$, a
contradiction.

Thus the claim holds.
\end{proof}

Now we choose $k,\ell$ such that\\
(1) $G[u_{-k},u_{\ell}]$ is $(u_{-k},u_0,u_{\ell})$-composed with
canonical ordering $u_{-k},u_{-k+1},\ldots,u_{\ell}$;\\
(2) any two nonadjacent vertices in $[u_{-k},u_{\ell}]$ have degree
sum less than $n$; and\\
(3) $k+\ell$ is as big as possible.

By Claim 7, we have that there exists a vertex $u_s\in
[u_{-k+1},u_{\ell}]$ such that $d(u_{-k-1})+d(u_s)\geq n$, or there
exists a vertex $u_s\in [u_{-k},u_{\ell-1}]$ such that
$d(u_s)+d(u_{\ell+1})\geq n$, or $d(u_{-k-1})+d(u_{\ell+1})\geq n$.
Thus, we have

\begin{claim}
$(u_{-k-1},u_{\ell})$ or $(u_{-k},u_{\ell+1})$ or
$(u_{-k-1},u_{\ell+1})$ is $u_0$-good on $C$.
\end{claim}

\begin{proof}
If there exists a vertex $u_s\in [u_{-k+1},u_{\ell}]$ such that
$d(u_{-k-1})+d(u_s)\geq n$, then, by Lemma 1, there exists a
$(u_0,u_{\ell})$-path $P$ such that $V(P)=[u_{-k},u_{\ell}]$, and
there exists a $(u_0u_{\ell},u_su_{-k})$-pair $D'$ such that
$V(D')=[u_{-k},u_{\ell}]$, and $D=D'+u_{-k}u_{-k-1}$ is a
$(u_0u_{\ell},u_su_{-k-1})$-pair such that
$V(D)=[u_{-k-1},u_{\ell}]$. Thus $(u_{-k-1},u_{\ell})$ is $u_0$-good
on $C$.

If there exists a vertex $u_s\in [u_{-k},u_{\ell-1}]$ such that
$d(u_s)+d(u_{\ell+1})\geq n$, we can prove the result similarly.

If $d(u_{-k-1})+d(u_{\ell+1})\geq n$, then by Lemma 1, there exists
a $(u_0,u_{\ell})$-path $P'$ such that $V(P')=[u_{-k},u_{\ell}]$ and
there exists a $(u_0,u_{-k})$-path $P''$ such that
$V(P'')=[u_{-k},u_{\ell}]$. Then $P=P'u_1u_{\ell+1}$ is a
$(u_0,u_{\ell+1})$-path such that $V(P)=[u_{-k},u_{\ell+1}]$, and
$D=P''u_{-k}u_{-k-1}\cup u_{\ell+1}$ is a
$(u_0u_{\ell+1},u_{\ell+1}u_{-k-1})$-pair such that
$V(D)=[u_{-k-1},u_{\ell+1}]$. Thus $(u_{-k-1},u_{\ell+1})$ is
$u_0$-good on $C$.
\end{proof}

\begin{claim}
There exist $v_{-k'}\in V(\overrightarrow{C}[v_{-1},u_{-k-1}])$ and
$v_{{\ell}'}\in V(\overleftarrow{C}[v_1,u_{\ell+1}])$ such that
$(v_{-k'},v_{{\ell}'})$ is is $v_0$-good on $C$.
\end{claim}

\begin{proof}
By Claim 6, we have $k\leq r_2-2$ and $l\leq r_1-2$.

If $v_1v_{-1}\notin E(G)$, then by Claim 2, $d(v_1)+d(v_{-1})\geq
n$. Then $P=v_0v_1$ is a $(v_0,v_1)$-path and $D=v_0v_{-1}\cup v_1$
is a $(v_0v_1,v_{-1}v_1)$-pair. Thus we have that $(v_{-1},v_1)$ is
$v_0$-good on $C$.

Now we assume that $v_1v_{-1}\in E(G)$, and then $G[v_{-1},v_1]$ is
$(v_{-1},v_0,v_1)$-composed.

Let $r'_2=r_2-k$ and $r'_1=r_1-\ell$.

\begin{subclaim}
If $G[v_{-k'},v_{{\ell}'}]$ is $(v_{-k'},v_0,v_{{\ell}'})$-composed
with canonical ordering $v_{-k'},v_{-k'+1},\ldots,$ $v_{{\ell}'}$,
then $k'\leq r'_2-1$ and ${\ell}'\leq r'_1-1$.
\end{subclaim}

\begin{proof}
Let $D_1,D_2,\ldots,D_r$ be a canonical sequence of
$G[v_{-k'},v_{{\ell}'}]$ corresponding to the canonical ordering
$v_{-k'},v_{-k'+1},\ldots,v_{{\ell}'}$. Suppose that $k'>r'_2-1$.
Consider the the graph $D'=D_{\widehat{-r'_2}}$, where
$\widehat{-r'_2}$ is the smallest integer such that $v_{-r'_2}\in
V(D_{\widehat{-r'_2}})$. Let $V(D')=[v_{-r'_2},v_{{\ell}''}]$. By
Lemma 1, there exists a $(v_0,v_{{\ell}''})$-path $P$ such that
$V(P)=[v_{-r'_2},u_{{\ell}''}]$. Then
$C'=P\overrightarrow{C}[u_{\ell},v_{{\ell}''}]P'R$ is a cycle which
contains all the vertices in $V(C)\cup V(R)$, a contradiction.
\end{proof}

Similarly to Claim 7, we have

\begin{subclaim}
If $G[v_{-k'},v_{{\ell}'}]$ is
$(v_{-k'},v_0,v_{{\ell}'})$-composed with canonical ordering
$v_{-k'},v_{-k'+1},$ $\ldots,v_{{\ell}'}$, where $k'\leq r'_2-1$ and
$\ell\leq r'_1-1$, and any two nonadjacent vertices in
$[v_{-k'-1},v_{{\ell}'+1}]$ have degree sum less than $n$, then one
of the following is true:\\
(1) $G[v_{-k'-1},v_{{\ell}'}]$ is
$(v_{-k'-1},v_0,v_{{\ell}'})$-composed
with canonical ordering $v_{-k'-1},v_{-k'},\ldots,v_{{\ell}'}$,\\
(2) $G[v_{-k'},v_{l'+1}]$ is $(v_{-k'},v_0,v_{{\ell}'+1})$-composed
with canonical ordering $v_{-k'},v_{-k'+1},\ldots,v_{{\ell}'+1}$,\\
(3) $G[v_{-k'-1},v_{l'+1}]$ is
$(v_{-k'-1},v_0,v_{{\ell}'+1})$-composed with canonical ordering
$v_{-k'-1},v_{-k'},\ldots,$ $v_{{\ell}'+1}$.
\end{subclaim}

Now we choose $k',{\ell}'$ such that\\
(1) $G[v_{-k'},v_{{\ell}'}]$ is $(v_{-k'},v_0,v_{{\ell}'})$-composed
with canonical ordering $v_{-k'},v_{-k'+1},\ldots,v_{{\ell}'}$;\\
(2) any two nonadjacent vertices in $[v_{-k'},v_{{\ell}'}]$ have
degree sum less than $n$; and\\
(3) $k'+{\ell}'$ is as big as possible.

Similarly to Claim 8, we have $(v_{-k'-1},v_{l'})$ or
$(v_{-k'},v_{l'+1})$ or $(v_{-k'-1},v_{l'+1})$ is $v_0$-good on $C$.
\end{proof}

From Claims 8 and 9, we get that there exists a cycle which
contains all the vertices in $V(C)\cup V(R)$ by Lemma 3, a
contradiction.

\begin{case}
$r=1$ and $u_0v_0\in E(G)$.
\end{case}

We have $u_0u_{-1}\in E(G)$ and $u_0u_{-r_2}\notin E(G)$, where
$u_{-r_2}=v_{-1}$. Let $u_{-k-1}$ be the first vertex in
$\overleftarrow{C}[u_{-1},v_{-1}]$ such that $u_0u_{-k-1}\notin
E(G)$. Then $k\leq r_2-1$.

Similarly, let $v_{\ell+1}$ be the first vertex in
$\overleftarrow{C}[v_1,u_1]$ such that $v_0v_{\ell+1}\notin E(G)$.
Then $\ell\leq r_1-1$.

\begin{claim}
Let $x\in[u_{-k-1},u_{-1}]$ and $y\in[v_1,v_{\ell+1}]$. Then
\begin{mathitem}
\item $xz_1,xv_0\notin \overline{E}(G)$,
\item $yz_1,yu_0\notin \overline{E}(G)$,
\item $xy\notin \overline{E}(G)$.
\end{mathitem}
\end{claim}

\begin{proof}
$(i)$ If $x=u_{-1}$, then by Claims 1 and 3, we have
$u_{-1}z_1,u_{-1}v_0\notin \overline{E}(G)$. So we {assume} that
$x=u_{-k'}$ where $-k'\in[-k-1,-2]$ and $u_0u_{-k'+1}\in E(G)$.

If $u_{-k'}z_1\in \overline{E}(G)$, then
$C'=u_0u_{-k'+1}\overrightarrow{C}[u_{-k'+1},u_{-1}]u_{-1}u_1\overrightarrow{C}[u_1,u_{-k'}]u_{-k'}z_1u_0$
is an $o$-cycle which contains all the vertices in $V(C)\cup V(R)$,
a contradiction.

If $u_{-k'}v_0\in \overline{E}(G)$, then
$C'=u_0u_{-k'+1}\overrightarrow{C}[u_{-k'+1},u_{-1}]u_{-1}u_1\overrightarrow{C}[u_1,v_1]v_1v_{-1}\overrightarrow{C}[v_{-1},u_{-k'}]$
$u_{-k'}v_0R$ is an $o$-cycle which contains all the vertices in
$V(C)\cup V(R)$, a contradiction.

The assertion $(ii)$ can be proved similarly.

$(iii)$ If $x=u_{-1}$ and $y=v_1$, then by Claim 3, we have
$xy\notin \overline{E}(G)$.

If $u_{-k'}v_1\in \overline{E}(G)$, where $k'\in[2,k+1]$, then
$C'=u_0R\overrightarrow{C}[v_0,u_{-k'}]u_{-k'}v_1\overleftarrow{C}[v_1,u_1]u_1u_{-1}$
$\overleftarrow{C}[u_{-1},u_{-k'+1}]u_{-k'+1}u_0$ is an $o$-cycle
which contains all the vertices in $V(C)\cup V(R)$, a contradiction.

If $u_{-1}v_{{\ell}'}\in \overline{E}(G)$, where
${\ell}'\in[2,\ell+1]$, then we can prove the result similarly.

If $u_{-k'}v_{{\ell}'}\in \overline{E}(G)$, where $k'\in[2,k+1]$ and
${\ell}'\in[2,\ell+1]$, then
$C'=u_0u_{-k'+1}\overrightarrow{C}[u_{-k'+1},u_{-1}]$
$u_{-1}u_1\overrightarrow{C}[u_1,v_{l'}]v_{l'}u_{-k'}\overleftarrow{C}[u_{-k'},v_{-1}]v_{-1}v_1\overleftarrow{C}[v_1,v_{l'-1}]v_{l'-1}v_0R$
is an $o$-cycle which contains all the vertices in $V(C)\cup V(R)$,
a contradiction.
\end{proof}

\begin{claim}
Either $u_{-k-1}u_0\notin \overline{E}(G)$, or
$v_{\ell+1}v_0\notin \overline{E}(G)$.
\end{claim}

\begin{proof}
Assume the opposite. Since $u_{-k-1}u_0,v_{\ell+1}v_0\notin E(G)$,
we have $d(u_{-k-1})+d(u_0)\geq n$ and $d(v_{\ell+1})+d(v_0)\geq n$.
By Claim 10, we have $d(u_0)+d(v_{\ell+1})<n$ and
$d(v_0)+d(u_{-k-1})<n$, a contradiction.
\end{proof}

Without loss of generality, we assume that $u_{-k-1}u_0\notin
\overline{E}(G)$. If $v_{\ell+1}v_0\notin \overline{E}(G)$, then the
subgraph induced by
$\{z_1,v_0,v_\ell,v_{\ell+1},u_0,u_{-k},u_{-k-1}\}$ is a $D$ which
is not heavy, a contradiction. Since $v_0v_{\ell+1}\notin E(G)$, we
have $d(v_0)+d(v_{\ell+1})\geq n$.

\begin{claim}
Either $(v_{-1},v_1)$ or $(v_{-1},v_{\ell+1})$ is $v_0$-good on $C$.
\end{claim}

\begin{proof}
If $v_1v_{-1}\notin E(G)$, then, by Claim 2, $d(v_1)+d(v_{-1})\geq
n$. Then $P=v_0v_1$ is a $(v_0,v_1)$-path and $D=v_0v_{-1}\cup v_1$
is a $(v_0v_1,v_{-1}v_1)$-pair. Thus, $(v_{-1},v_1)$ is
$v_0$-good on $C$.

If $v_1v_{-1}\in E(G)$, then
$P=v_0v_{\ell}\overrightarrow{C}[v_{\ell},v_1]v_1v_{-1}$ is a
$(v_0,v_{-1})$-path and $D=v_0\cup
v_{-1}v_1\overleftarrow{C}[v_1,v_{\ell+1}]$ is a
$(v_0v_{-1},v_0v_{l+1})$-pair. Since $d(v_0)+d(v_{\ell+1})\geq n$,
we have that $(v_{-1},v_{\ell+1})$ is $v_0$-good on $C$.
\end{proof}

\begin{claim}
If $G[u_{-k'},u_{{\ell}'}]$ is $(u_{-k'},u_0,u_{{\ell}'})$-composed
with canonical ordering $u_{-k'},u_{-k'+1},\ldots,$ $u_{{\ell}'}$,
then $k'\leq r_2-2$ and ${\ell}'\leq r_1-\ell-1$.
\end{claim}

\begin{proof}
The claim can be proved similarly as Claims 6 and 9.1.
\end{proof}

Now we prove the following claim.

\begin{claim}
If $G[u_{-k'},u_{{\ell}'}]$ is
$(u_{-k'},u_0,u_{{\ell}'})$-composed with canonical ordering
$u_{-k'},u_{-k'+1},\ldots,u_{{\ell}'}$, where $k'\leq r_2-2$ and
${\ell}'\leq r_1-\ell-1$, and any two nonadjacent vertices in
$[u_{-k'-1},u_{{\ell}'+1}]$ have degree sum less than $n$, then one
of the following is true:\\
(1) $G[u_{-k'-1},u_{{\ell}'}]$ is
$(u_{-k'-1},u_0,u_{{\ell}'})$-composed
with canonical ordering $u_{-k'-1},u_{-k'},\ldots,u_{{\ell}'}$,\\
(2) $G[u_{-k'},u_{{\ell}'+1}]$ is
$(u_{-k'},u_0,u_{{\ell}'+1})$-composed
with canonical ordering $u_{-k'},u_{-k'+1},\ldots,u_{{\ell}'+1}$,\\
(3) $G[u_{-k'-1},u_{{\ell}'+1}]$ is
$(u_{-k'-1},u_0,u_{{\ell}'+1})$-composed with canonical ordering
$u_{-k'-1},u_{-k'},\ldots,u_{{\ell}'+1}$.
\end{claim}

\begin{proof}
Assume the opposite, which implies that for every vertex
$u_s\in[u_{-k'+1},u_{{\ell}'}]$, $u_{-k'-1}u_s\notin E(G)$, and for
every vertex $u_s\in[u_{-k'},u_{{\ell}'-1}]$, we have
$u_{{\ell}'+1}u_s\notin E(G)$ and $u_{-k'-1}u_{{\ell}'+1}\notin
E(G)$.

\begin{subclaim}
Let $v\in\{v_0,v_1\}$ and
$u_s\in[u_{-k'-1},u_{{\ell}'+1}]\backslash\{u_0\}$. Then $vu_s\notin
\overline{E}(G)$.
\end{subclaim}

\begin{proof}
Similarly to Claim 7.1, we have $v_0u_s\notin \overline{E}(G)$.

Now we assume that $v_1u_s\in \overline{E}(G)$.

Note that if $v_0v_2\notin E(G)$, then $d(v_0)+d(v_2)\geq n$. We
have $v_0v_2\in \overline{E}(G)$.

If $s\in[-k'-1,-2]$, then $s+1\in[-k',-1]$. By the definition of a
composed graph, there exists $t\in [1,{\ell}']$ such that $G[u_{s+1},u_t]$
is $(u_{s+1},u_0,u_t)$-composed. By Lemma 1, there exists a
$(u_0,u_t)$-path $P$ such that $V(P)=[u_{s+1},u_t]$. Then
$C'=P\overrightarrow{C}[u_t,v_1]v_1u_s\overleftarrow{C}[u_s,v_0]R$
is an $o$-cycle which contains all the vertices in $V(C)\cup V(R)$,
a contradiction.

If $s=-1$, by Claim 3, we have $v_1u_{-1}\notin \overline{E}(G)$.

If $s=1$, then
$C'=\overleftarrow{C}[u_0,v_{-1}]v_{-1}v_1u_1\overrightarrow{C}[u_1,v_2]v_2v_0R$
is an $o$-cycle which contains all the vertices in $V(C)\cup V(R)$,
a contradiction.

If $s\in[2,{\ell}'+1]$, then $s-1\in[1,{\ell}']$. By the definition
of a composed graph, there exists $t\in [-k',-1]$ such that $G[u_t,u_{s-1}]$
is $(u_t,u_0,u_{s-1})$-composed. By Lemma 1, there exists a
$(u_0,u_t)$-path $P$ such that $V(P)=[u_t,u_{s-1}]$. Then
$C'=P\overleftarrow{C}[u_t,v_{-1}]v_{-1}v_1u_s\overrightarrow{C}[u_s,v_2]v_2v_0R$
is an $o$-cycle which contains all the vertices in $V(C)\cup V(R)$,
a contradiction.
\end{proof}

Let $G'=G[[u_{-k'-1},u_{{\ell}'}]\cup\{v_0,v_1\}]$ and
$G''=G[[u_{-k'-1},u_{{\ell}'+1}]\cup\{v_0,v_1\}]$. Then, similarly
to Claim 7.2, we have

\begin{subclaim}
$G''$, and then $G'$, is $\{K_{1,3},N_{1,1,2}\}$-free.
\end{subclaim}

Thus we can prove the claim similarly to Claim 7.
\end{proof}

Now we choose $k',{\ell}'$ such that\\
(1) $G[u_{-k'},u_{{\ell}'}]$ is $(u_{-k'},u_0,u_{{\ell}'})$-composed;\\
(2) any two nonadjacent vertices in $[u_{-k'},u_{{\ell}'}]$ have
degree sum less than $n$; and\\
(3) $k'+{\ell}'$ is as big as possible.

Similarly to Claim 8, we have

\begin{claim}
$(u_{-k'-1},u_{{\ell}'})$ or $(u_{-k'},u_{{\ell}'+1})$ or
$(u_{-k'-1},u_{{\ell}'+1})$ is $u_0$-good on $C$.
\end{claim}

By Claim 13, we have $k'\leq r_2-2$ and ${\ell}'\leq r_1-\ell-2$.

From Claims 12 and 15, we can get that there exists a cycle which
contains all vertices in $V(C)\cup V(R)$ by Lemma 3, a
contradiction.

The proof is complete.

\section{Proof of Theorem 9}

Let $C$ be a longest cycle of $G$ with a given orientation. We use
$n$ to denote the order of $G$, and $c$ the length of $C$. Assume
that $G$ is not Hamiltonian. Then $V(G)\backslash V(C)\neq
\emptyset$. Since $G$ is 2-connected, there exists a
$(u_0,v_0)$-path of length at least 2 which is internally disjoint
with $C$, where $u_0,v_0\in V(C)$. Let $R=z_0z_1z_2\cdots z_{r+1}$
be such a path {which is} as short as possible, where $z_0=u_0$ and
$z_{r+1}=v_0$. We assume that the length of
$\overrightarrow{C}[u_0,v_0]$ is $r_1+1$ and the length of
$\overrightarrow{C}[v_0,u_0]$ is $r_2+1$, where $r_1+r_2+2=c$. We
denote the vertices of $C$ by $\overrightarrow{C}=u_0u_1u_2\cdots
u_{r_1}v_0u_{-r_2}u_{-r_2+1}\cdots u_{-1}u_0$ or
$\overleftarrow{C}=v_0v_1v_2\cdots
v_{r_1}u_0v_{-r_2}v_{-r_2+1}\cdots v_{-1}v_0$, where
$u_{\ell}=v_{r_1+1-\ell}$ and $u_{-k}=v_{-r_2-1+k}$. Let $H$ be the
component of $G-C$ which contains the vertices in $[z_1,z_r]$.

As in Section 3, we have the following claims.

\setcounter{claim}{0}
\begin{claim}
Let $x\in V(H)$ and $y\in \{v_1,v_{-1},u_1,u_{-1}\}$. Then $xy\notin
\overline{E}(G)$.
\end{claim}

\begin{claim}
$u_1u_{-1}\in \overline{E}(G)$, $v_1v_{-1}\in \overline{E}(G)$.
\end{claim}

\begin{claim}
$u_1v_{-1}\notin \overline{E}(G)$, $u_{-1}v_1\notin
\overline{E}(G)$, $u_0v_{\pm 1}\notin \overline{E}(G),u_{\pm
1}v_0\notin \overline{E}(G)$.
\end{claim}

\begin{claim}
Either $u_1u_{-1}$ or $v_1v_{-1}$ is in $E(G)$.
\end{claim}

By Claim 4, without {loss} of generality, we assume that
$u_1u_{-1}\in E(G)$. {Then} we have that $G[u_{-1},u_1]$ is
$(u_{-1},u_0,u_1)$-composed.

\begin{claim}
If $G[u_{-k},u_{\ell}]$ is $(u_{-k},u_0,u_{\ell})$-composed, then
$k\leq r_2-2$ and $\ell\leq r_1-2$.
\end{claim}

The proof of Claim 5 is similar to that of Claim 6 in Section 3.

Now we {prove} the following claim.

\begin{claim}
If $G[u_{-k},u_{\ell}]$ is $(u_{-k},u_0,u_{\ell})$-composed with
canonical ordering $u_{-k},u_{-k+1},\ldots,u_{\ell}$, where $k\leq
r_2-2$ and $l\leq r_1-2$, and any two nonadjacent vertices in
$[u_{-k-1},u_{\ell+1}]$ have degree sum less than $n$, then one of
the following is true:\\
(1) $G[u_{-k-1},u_{\ell}]$ is $(u_{-k-1},u_0,u_{\ell})$-composed
with canonical ordering $u_{-k-1},u_{-k},\ldots,u_{\ell}$,\\
(2) $G[u_{-k},u_{\ell+1}]$ is $(u_{-k},u_0,u_{\ell+1})$-composed
with canonical ordering $u_{-k},u_{-k+1},\ldots,u_{\ell+1}$,\\
(3) $G[u_{-k-1},u_{\ell+1}]$ is $(u_{-k-1},u_0,u_{\ell+1})$-composed
with canonical ordering $u_{-k-1},u_{-k},\ldots,u_{\ell+1}$.
\end{claim}

\begin{proof}
Assume the opposite, which implies that for every vertex
$u_s\in[u_{-k+1},u_{\ell}]$, $u_{-k-1}u_s\notin E(G)$, and for every
vertex $u_s\in[u_{-k},u_{\ell-1}]$, $u_{\ell+1}u_s\notin E(G)$ and
$u_{-k-1}u_{\ell+1}\notin E(G)$.

\begin{subclaim}
For every vertex $z\in\{z_1,z_2\}$ and
$u_s\in[u_{-k-1},u_{\ell+1}]\backslash\{u_0\}$ we have $zu_s\notin
\overline{E}(G)$; and if $z_2u_0\notin E(G)$, then also $z_2u_0\notin
\overline{E}(G)$.
\end{subclaim}

This claim can be proved similarly to Claims 5 and 7.1 in Section 3.

Let $G'=G[[u_{-k-1},u_{\ell}]\cup\{z_1,z_2\}]$ and
$G''=G[[u_{-k-1},u_{\ell+1}]\cup\{z_1,z_2\}]$.

Similarly to Claims 7.2 and 7.3 in Section 3, we have

\begin{subclaim} $G''$, and then
$G'$, is $\{K_{1,3},N_{1,1,2},H_{1,1}\}$-free.
\end{subclaim}

\begin{subclaim}
$N_{G'}(u_0)\backslash\{z_1,z_2\}$ is a clique.
\end{subclaim}

Now, we define $N_i=\{x\in V(G'): d_{G'}(x,u_{-k-1})=i\}$. Then we
have $N_0=\{u_{-k-1}\}$, $N_1=\{u_{-k}\}$ and
$N_2=N_{G'}(u_{-k})\backslash\{u_{-k-1}\}$.

By the definition of a composed graph, we have that $|N_2|\geq 2$. If there
are two vertices $x,x'\in N_2$ such that $xx'\notin E(G')$, then the
graph induced by $\{u_{-k},u_{-k-1},x,x'\}$ is a claw. Thus
$N_2$ is a clique.

We assume $u_0\in N_j$, where $j\geq 2$. Then $z_1\in N_{j+1}$; and
$z_2\in N_{j+1}$ if $z_2u_0\in E(G)$ and $z_2\in N_{j+2}$ if
$z_2u_0\notin E(G)$.

If $|N_i|=1$ for some $i\in[2,j-1]$, say, $N_i=\{x\}$, then $x$ is a
cut vertex of the graph $G[u_{-k},u_{\ell}]$. By the definition of a
composed graph, $G[u_{-k},u_{\ell}]$ is 2-connected. This implies
$|N_i|\geq 2$ for every $i\in[2,j-1]$.

\begin{subclaim}
For $i\in [1,j]$, $N_i$ is a clique.
\end{subclaim}

\begin{proof}
If $i<j$, or $i=j$ and $z_2u_0\notin E(G)$, then we can prove the
assertion similarly
to Claim 7.4 in Section 3. Thus we assume that
$i=j$ and $z_2u_0\in E(G)$.

If $j=2$, the assertion is true by the analysis above. So we assume
that $j\geq 3$, and we have that $N_{j-3},N_{j-2},N_{j-1},N_{j+1}$ is
nonempty and $|N_{j-1}|\geq 2$.

First we prove that for every $x\in N_j\backslash\{u_0\}$, $u_0x\in
E(G)$. We assume that $u_0x\notin E(G)$.

By Claim 6.1 we have $xz_1\notin E(G)$. If $u_0$ and $x$ have a
common neighbor in $N_{j-1}$, denoted $w$, then let $v$ be a
neighbor of $w$ in $N_{j-2}$; but then the graph induced by
$\{w,v,u_0,x\}$ is a claw, a contradiction. Thus we have that $u_0$
and $x$ have no common neighbors in $N_{j-1}$.

Let $w$ be a neighbor of $u_0$ in $N_{j-1}$ and $w'$ be a neighbor
of $x$ in $N_{j-1}$. Then $u_0w',xw\notin E(G)$. Let $v$ be a
neighbor of $w$ in $N_{j-2}$ and $u$ be a neighbor of $v$ in
$N_{j-3}$. If $w'v\notin E(G)$, then the graph induced by
$\{w,v,w',u_0\}$ is a claw, a contradiction. Thus we have $w'v\in
E(G)$, and then the graph induced by $\{v,u,w',x,w,u_0,z_1\}$ is an
$N_{1,1,2}$, a contradiction.

Thus we have $u_0x\in E(G)$ for every $x\in N_j$. Then by Claim 6.3,
we have that $N_j$ is a clique.
\end{proof}

If there exists some vertex $y\in N_{j+1}$ other than $z_1$ and
$z_2$, then we have $yu_0\notin E(G)$ by Claim 6.3. Let $x$ be a
neighbor of $y$ in $N_j$, $w$ be a neighbor of $u_0$ in $N_{j-1}$
and $v$ be a neighbor of $w$ in $N_{j-2}$. Then $xu_0\in E(G)$ by
Claim 6.4 and $xw\in E(G)$ by Claim 6.3. Thus the graph induced by
$\{w,v,x,y,u_0,z_1,z_2\}$ is an $N_{1,1,2}$ if $z_2u_0\notin E(G)$,
and is an $H_{1,1}$ if $z_2u_0\in E(G)$, a contradiction. So we
assume that all vertices in $[u_{-k},u_{\ell}]$ are in
$\bigcup_{i=1}^jN_i$.

If $u_{\ell}\in N_j$, then let $w$ be a neighbor of $u_0$ in
$N_{j-1}$ and $v$ be a neighbor of $w$ in $N_{j-2}$. Then the graph
induced by $\{w,v,u_0,z_1,u_{\ell},u_{\ell+1}\}$ is an $N_{1,1,2}$
if $z_2u_0\notin E(G)$, and is an $H_{1,1}$ if $z_2u_0\in E(G)$, a
contradiction. Thus we have that $u_{\ell}\notin N_j$ and then
$j\geq 3$.

Let $u_{\ell}\in N_i$, where $i\in[2,j-1]$. If $u_{\ell}$ has a
neighbor in $N_{i+1}$, then let $y$ be a neighbor of $u_{\ell}$ in
$N_{i+1}$, and $w$ be a neighbor of $u_{\ell}$ in $N_{i-1}$. Then
the graph induced by $\{u_{\ell},w,y,u_{\ell+1}\}$ is a claw, a
contradiction. Thus we have that $u_{\ell}$ has no neighbors in
$N_{i+1}$.

Let $x\in N_i$ be a vertex other than $u_{\ell}$ which has a
neighbor $y$ in $N_{i+1}$ such that it has a neighbor $z$ in
$N_{i+2}$. Let $w$ be a neighbor of $x$ in $N_{i-1}$, and $v$ be a
neighbor of $w$ in $N_{i-2}$. If $u_{\ell}w\notin E(G)$, then the
graph induced by $\{x,w,u_{\ell},y\}$ is a claw, a contradiction.
Thus we have that that $u_{\ell}w\in E(G)$. Then the graph induced
by $\{w,v,u_{\ell},u_{\ell+1},x,y,z\}$ is an $N_{1,1,2}$, a
contradiction.

Thus the claim holds.
\end{proof}

Now we choose $k,\ell$ such that\\
(1) $G[u_{-k},u_{\ell}]$ is $(u_{-k},u_0,u_{\ell})$-composed with
canonical ordering $u_{-k},u_{-k+1},\ldots,u_{\ell}$;\\
(2) any two nonadjacent vertices in $[u_{-k},u_{\ell}]$ have degree
sum less than $n$; and\\
(3) $k+\ell$ is as big as possible.

Similarly to Claims 8 and 9 in Section 3, we have

\begin{claim}
$(u_{-k-1},u_{\ell})$ or $(u_{-k},u_{\ell+1})$ or
$(u_{-k-1},u_{\ell+1})$ is $u_0$-good on $C$.
\end{claim}

\begin{claim}
There exist $v_{-k'}\in V(\overrightarrow{C}[v_{-1},u_{-k-1}])$ and
$v_{{\ell}'}\in V(\overleftarrow{C}[v_1,u_{\ell+1}])$ such that
$(v_{-k'},v_{{\ell}'})$ is $v_0$-good on $C$.
\end{claim}

From Claims 7 and 8, we can get that there exists a cycle which
contains all the vertices in $V(C)\cup V(R)$ by Lemma 3, a
contradiction.

The proof is complete.

\end{document}